\documentclass[12pt]{amsart}
\usepackage{amsmath,amssymb,amsbsy,amsfonts,amsthm,latexsym,mathabx,
            amsopn,amstext,amsxtra,euscript,amscd,stmaryrd,mathrsfs,
            cite,array,mathtools,enumerate}
\usepackage{url}
\usepackage[colorlinks,linkcolor=blue,anchorcolor=blue,citecolor=blue]{hyperref}

\usepackage{color}
\usepackage{float}

\hypersetup{breaklinks=true}

\begin{document}

\newtheorem{theorem}{Theorem}
\newtheorem{lemma}[theorem]{Lemma}
\newtheorem{claim}[theorem]{Claim}
\newtheorem{cor}[theorem]{Corollary}
\newtheorem{proposition}[theorem]{Proposition}
\newtheorem{definition}{Definition}
\newtheorem{question}[theorem]{Question}
\newtheorem{remark}[theorem]{Remark}
\newcommand{\hh}{{{\mathrm h}}}

\numberwithin{equation}{section}
\numberwithin{theorem}{section}
\numberwithin{table}{section}

\def\sssum{\mathop{\sum\!\sum\!\sum}}
\def\ssum{\mathop{\sum\ldots \sum}}
\def\dsum{\mathop{\sum \sum}}
\def\iint{\mathop{\int\ldots \int}}

\def\squareforqed{\hbox{\rlap{$\sqcap$}$\sqcup$}}
\def\qed{\ifmmode\squareforqed\else{\unskip\nobreak\hfil
\penalty50\hskip1em\null\nobreak\hfil\squareforqed
\parfillskip=0pt\finalhyphendemerits=0\endgraf}\fi}%%

%  use the AMS-Euler Fraktur fonts
%%%%%%%%%%%%%%%%%%%%%%%%%%%%%%%%%%
\newfont{\teneufm}{eufm10}
\newfont{\seveneufm}{eufm7}
\newfont{\fiveeufm}{eufm5}
%%%%%%%%%%%%%%%%%%%%%%%%%%%%%%%%%
%
%  allow automatic size selection in math mode
%
%%%%%%%%%%%%%%%%%%%%%%%%%%%%%%%%%
\newfam\eufmfam
     \textfont\eufmfam=\teneufm
\scriptfont\eufmfam=\seveneufm
     \scriptscriptfont\eufmfam=\fiveeufm
%%%%%%%%%%%%%%%%%%%%%%%%%%%%%%%%%
%
%  \frak works on a single symbol at a time...
%
\def\frak#1{{\fam\eufmfam\relax#1}}

\newcommand{\bflambda}{{\boldsymbol{\lambda}}}
\newcommand{\bfmu}{{\boldsymbol{\mu}}}
\newcommand{\bfxi}{{\boldsymbol{\xi}}}
\newcommand{\bfrho}{{\boldsymbol{\rho}}}

\def\fK{\mathfrak K}
\def\fT{\mathfrak{T}}

\def\fA{{\mathfrak A}}
\def\fB{{\mathfrak B}}
\def\fC{{\mathfrak C}}

\def\E{\mathsf {E}}

\def \balpha{\bm{\alpha}}
\def \bbeta{\bm{\beta}}
\def \bgamma{\bm{\gamma}}
\def \blambda{\bm{\lambda}}
\def \bchi{\bm{\chi}}
\def \bphi{\bm{\varphi}}
\def \bpsi{\bm{\psi}}

\def\eqref#1{(\ref{#1})}

\def\vec#1{\mathbf{#1}}

%\def\squareforqed{\hbox{\rlap{$\sqcap$}$\sqcup$}}
%\def\qed{\ifmmode\squareforqed\else{\unskip\nobreak\hfil
%\penalty50\hskip1em\null\nobreak\hfil\squareforqed
%\parfillskip=0pt\finalhyphendemerits=0\endgraf}\fi}

%%%%%%%%%%%%%%%%%%%%%%%%%
% Alphabet calligraphie %
%%%%%%%%%%%%%%%%%%%%%%%%%
\def\cA{{\mathcal A}}
\def\cB{{\mathcal B}}
\def\cC{{\mathcal C}}
\def\cD{{\mathcal D}}
\def\cE{{\mathcal E}}
\def\cF{{\mathcal F}}
\def\cG{{\mathcal G}}
\def\cH{{\mathcal H}}
\def\cI{{\mathcal I}}
\def\cJ{{\mathcal J}}
\def\cK{{\mathcal K}}
\def\cL{{\mathcal L}}
\def\cM{{\mathcal M}}
\def\cN{{\mathcal N}}
\def\cO{{\mathcal O}}
\def\cP{{\mathcal P}}
\def\cQ{{\mathcal Q}}
\def\cR{{\mathcal R}}
\def\cS{{\mathcal S}}
\def\cT{{\mathcal T}}
\def\cU{{\mathcal U}}
\def\cV{{\mathcal V}}
\def\cW{{\mathcal W}}
\def\cX{{\mathcal X}}
\def\cY{{\mathcal Y}}
\def\cZ{{\mathcal Z}}
\newcommand{\rmod}[1]{\: \mbox{mod} \: #1}

\def\cg{{\mathcal g}}

\def\e{{\mathbf{\,e}}}
\def\ep{{\mathbf{\,e}}_p}
\def\eq{{\mathbf{\,e}}_q}

\def\em{{\mathbf{\,e}}_m}

\def\Tr{{\mathrm{Tr}}}
\def\Nm{{\mathrm{Nm}}}

\def\rE{{\mathrm{E}}}
\def\rT{{\mathrm{T}}}

 \def\SS{{\mathbf{S}}}

\def\lcm{{\mathrm{lcm}}}

\def\t{\tilde}
\def\ov{\overline}

\def\({\left(}
\def\){\right)}
\def\l|{\left|}
\def\r|{\right|}
\def\fl#1{\left\lfloor#1\right\rfloor}
\def\rf#1{\left\lceil#1\right\rceil}
\def\flq#1{\langle #1 \rangle_q}

\def\mand{\qquad \mbox{and} \qquad}

\newcommand{\commIg}[1]{\marginpar{%
\begin{color}{magenta}
\vskip-\baselineskip %raise the marginpar a bit
\raggedright\footnotesize
\itshape\hrule \smallskip Ig: #1\par\smallskip\hrule\end{color}}}

\newcommand{\commSi}[1]{\marginpar{%
\begin{color}{blue}
\vskip-\baselineskip %raise the marginpar a bit
\raggedright\footnotesize
\itshape\hrule \smallskip Si: #1\par\smallskip\hrule\end{color}}}

%%%%%%%%%%%%%%%%%%%%%%%%%%%%%%%%%%%%%%%%%%%%%%%%%%%%%%%%
%%%%%%%%%%%%%%%%%%%%%%%%%%%%%%%%%%%%%%%%%%%%%%%%%%%%%%%%
%%%%%%%%%%%%%%%%%%%%%%%%%%%%%%%%%%%%%%%%%%%%%%%%%%%%%%%%
%%%%%%%%%%%%%%%%%%%%%%%%%%%%%%%%%%%%%%%%%%%%%%%%%%%%%%%%

%%%%%%%  END OF STANDARD STUFF %%%%%%%%%

%%%%%%%%%%%%%%%%%%%%%%%%%%%%%%%%%%%%%%%%%%%%%%%%%%%%%%%%
%%%%%%%%%%%%%%%%%%%%%%%%%%%%%%%%%%%%%%%%%%%%%%%%%%%%%%%%
%%%%%%%%%%%%%%%%%%%%%%%%%%%%%%%%%%%%%%%%%%%%%%%%%%%%%%%%
%%%%%%%%%%%%%%%%%%%%%%%%%%%%%%%%%%%%%%%%%%%%%%%%%%%%%%%
%%%%%%%%%%%
%%% Spell

\hyphenation{re-pub-lished}

\mathsurround=1pt

\def\bfdefault{b}
\overfullrule=5pt

\def \F{{\mathbb F}}
\def \K{{\mathbb K}}
\def \Z{{\mathbb Z}}
\def \Q{{\mathbb Q}}
\def \R{{\mathbb R}}
\def \C{{\\mathbb C}}
\def\Fp{\F_p}
\def \fp{\Fp^*}

\def\Smn{S_{k,\ell,q}(m,n)}

\def\Kmn{\cK_p(m,n)}
\def\psmn{\psi_p(m,n)}

\def\SM{\cS_{k,\ell,q}(\cM)}
\def\SMN{\cS_{k,\ell,q}(\cM,\cN)}
\def\SAMN{\cS_{k,\ell,q}(\cA;\cM,\cN)}
\def\SABMN{\cS_{k,\ell,q}(\cA,\cB;\cM,\cN)}

\def\SIJq{\cS_{k,\ell,q}(\cI,\cJ)}
\def\SAJq{\cS_{k,\ell,q}(\cA;\cJ)}
\def\SABJq{\cS_{k,\ell,q}(\cA, \cB;\cJ)}

\def\sM{\cS_{k,q}^*(\cM)}
\def\sMN{\cS_{k,q}^*(\cM,\cN)}
\def\sAMN{\cS_{k,q}^*(\cA;\cM,\cN)}
\def\sABMN{\cS_{k,q}^*(\cA,\cB;\cM,\cN)}

\def\sIJq{\cS_{k,q}^*(\cI,\cJ)}
\def\sAJq{\cS_{k,q}^*(\cA;\cJ)}
\def\sABJq{\cS_{k,q}^*(\cA, \cB;\cJ)}
\def\sABJp{\cS_{k,p}^*(\cA, \cB;\cJ)}

 \def \xbar{\overline x}

 \author[S.  Macourt] {Simon Macourt}
\address{Department of Pure Mathematics, University of New South Wales,
Sydney, NSW 2052, Australia}
\email{s.macourt@student.unsw.edu.au}

\begin{abstract}
We improve an existing result on exponential quadrilinear sums in the case of sums over multiplicative subgroups of a finite field and use it to give a new bound on exponential sums with quadrinomials.
\end{abstract}

\title{Bounds On Exponential Sums With Quadrinomials}

\maketitle

\section{Introduction}
\subsection{Set Up}
For a prime $p$, we use $\F_p$ to denote  the finite field of $p$ elements.

For a $t$-sparse polynomial  
$$
\Psi(X) = \sum_{i=1}^t a_i X^{k_i}
$$
with some   pairwise distinct non-zero integer exponents $k_1, \ldots, k_t$ and 
coefficients  $a_1, \ldots, a_t\in \F_p^*$,  and a multiplicative character $\chi$ of $\F_p^*$ we define  the   sums
$$
S_\chi(\Psi) = \sum_{x\in \F_p^*} \chi(x) \ep(\Psi(x)), 
$$
where $\ep(u) = \exp(2 \pi i u/p)$ and $\chi$ is an arbitrary 
multiplicative character of $\F_p^*$. The challenge for such sums is to provide a bound that is stronger than the Weil bound
\begin{align*}
S_\chi(\Psi)\le \max\{k_1,\dots,k_t\}p^{1/2},
\end{align*}
see \cite[Appendix 5, Example 12]{Weil}, by taking advantage of the arithmetic structure of the exponents. The case of exponential sums of monomials has seen much study with Shparlinski \cite{Shp1} providing the first such bound. Further improvements have been made by various other authors, see \cite{BGK, Bourg1, HBK, Kon, Shkr1, Sht}.
We also mention that Cochrane, Coffelt and Pinner, as well as others, have given several bounds on exponential sums with sparse polynomials, see \cite{CoCoPi1,CoCoPi2,CoPi1, CoPi2, CoPi3, CoPi4} and references therein, some of which we outline in Section \ref{prev}.

Here we provide some new bounds on quadrinomial exponential sums using the techniques in \cite{MaShkShp}. We thus define
\begin{equation} \label{eq:Quad}
\Psi(X)=aX^k+bX^\ell+cX^m+dX^n.
\end{equation} 
We mention that all our results extend naturally to more general sums with polynomials of the shape 
\begin{equation*} 
\Psi(X)=aX^k+f(X^\ell)+g(X^m)+h(X^n)
\end{equation*} 
for polynomials $f,g,h \in \F_p[X]$.

The notation $A \ll B$ is equivalent to $|A|\le c|B|$ for some constant $c$.
\subsection{Previous Results} \label{prev}
We compare our result for quadrinomials \eqref{eq:Quad} to those of Cochrane, Coffelt and Pinner \cite[Theorem 1.1]{CoCoPi1}
\begin{equation*}
S_\chi (\Psi)\ll \left(\frac{k \ell m n}{\max(k,\ell, m,n)}\right)^{1/9}p^{8/9}
\end{equation*}
which is non-trivial for 
$$\frac{k\ell mn}{\max(k,\ell , m, n)} < p,$$
and of Cochrane and Pinner \cite[Theorem 1.1]{CoPi1}
\begin{equation*}
S_\chi (\Psi)\ll (k\ell mn)^{1/16}p^{7/8}
\end{equation*}
which is non-trivial for $k \ell mn<p^2$.
Our new result in Theorem \ref{thm:Bound3} is independent of the size of the exponents but instead depends on various greatest common divisors.

\subsection{Main Result}
Our main result is the following theorem.
\begin{theorem}
\label{thm:Bound3}   
Let $\Psi(X)$ be a quadrinomial of the form~\eqref{eq:Quad} 
with $a,b,c,d  \in \F_p^*$.  
Define
$$\alpha= \gcd(k,p-1), \ \beta = \gcd(\ell ,p-1), \  \gamma =  \gcd(m,p-1), \ \delta=\gcd(n,p-1)
$$
and
$$
f =\frac{\alpha}{\gcd(\alpha,\delta)},\qquad g =\frac{\beta}{\gcd(\beta,\delta)}, \qquad h=\frac{\gamma }{\gcd(\gamma,\delta)}.
$$
Suppose $f\ge g\ge h$, then $p/\delta\ge f$ and

\begin{align*}
S_\chi(\Psi) \ll &pg^{-1/8} \\
&+\left\{
\begin{array}{ll}
p^{15/16}\delta^{1/32},& \text{if $ g\ge p^{1/2}\log p$},\\
p^{31/32}\delta^{1/32}g^{-1/16+o(1)},& \text{if $ f \ge p^{1/2}\log p>g$}, \\
p\delta^{1/32}(fg)^{-1/16+o(1)},& \text{if $p/\delta  \ge p^{1/2}\log p>f$}, \\
p^{31/32+o(1)}\delta^{3/32}(fg)^{-1/16},& \text{if $ p/\delta < p^{1/2}\log p$}.
\end{array}
\right.
\end{align*}
\end{theorem}
We mention that our result is independent of the size of our powers $k,l,m,n$ and is strongest when $\delta$ is small and $f,g,h$ are large. As mentioned in the previous section, previous results become trivial for quadrinomials of large degree. It is easy to see that our bound is non-trivial and improves previous results for a wide range of exponents $k, \ell, m$ and $n$.
\section{Preliminaries}
We recall the following classical  bound of  bilinear sums, 
see, for example,~\cite[Equation~1.4]{BouGar} or~\cite[Lemma~4.1]{Gar}.

\begin{lemma}
\label{lem:bilin} 
For any sets $\cX, \cY \subseteq \F_p$ and any  $\alpha= (\alpha_{x})_{x\in \cX}$, $\beta = \( \beta_{y}\)_{y \in \cY}$, 
with 
$$
\sum_{x\in \cX}|\alpha_{x}|^2 = A \mand  \sum_{y \in \cY}|\beta_{y}|^2 = B, 
$$
we have 
$$
\left |\sum_{x \in \cX}\sum_{y \in \cY} \alpha_{x} \beta_{y}  \ep(xy) \right| \le \sqrt{pAB}.
$$
\end{lemma}

We define $D_\times(\cU)$ to be the number of solutions of 
$$
(u_1-v_1)(u_2-v_2) = (u_3-v_3)(u_4-v_4), \qquad u_i,v_i \in \cU,\  i=1,2,3,4.
$$
We also define the multiplicative energy $\rE^\times (\cU, \cV)$ to be the number of solutions of
$$
u_1v_1=u_2v_2  \qquad u_i \in \cU,\ v_i \in \cV, \  i=1,2.
$$
When $\cU=\cV$, we write $\rE^\times (\cU, \cU)=E^\times (\cU)$.

We need the following result from \cite[Corollary 3.3]{MaShkShp}.
\begin{lemma}
\label{Bound Dx2}
For a multiplicative subgroup $\cG \subset \F_p^*$, we have
$$
D_\times(\cG) \ll   \left\{
\begin{array}{ll}
|\cG|^8 p^{-1} , & \text{if $|\cG| \ge p^{1/2}\log p$},\\
|\cG|^6 \log |\cG|, & \text{if $|\cG|< p^{1/2}\log p$}. 
\end{array}
\right.
$$
\end{lemma}
We also use \cite[Corollary 4.1]{MaShkShp}.
\begin{lemma}
\label{lem:EnergyShiftSubgr}
Let $\cG$ be a multiplicative subgroup of $\F_p^*$. 
Then for any   $\lambda \in \F_p^* $, we have 
$$
\rE^\times (\cG + \lambda) -\frac{|\cG|^4}{p} \ll  \left\{
\begin{array}{ll}
p^{1/2} |\cG|^{3/2},& \text{if $ |\cG| \ge p^{2/3}$},\\
|\cG|^3 p^{-1/2} , & \text{if $p^{2/3} > |\cG| \ge p^{1/2}\log p$},\\
|\cG|^2 \log |\cG|, & \text{if $|\cG|< p^{1/2}\log p$}. 
\end{array}
\right.
$$
\end{lemma}
We immediately obtain the following result by observing the dominant term from Lemma  \ref{lem:EnergyShiftSubgr}.
\begin{cor} \label{cor:MultEnergy}
Let $\cG$ be a multiplicative subgroup of $\F_p^*$. 
Then for any   $\lambda \in \F_p^* $, we have 
$$
\rE^\times (\cG + \lambda)\ll  \left\{
\begin{array}{ll}
|\cG^4|/p,& \text{if $ |\cG| \ge p^{1/2}\log p$},\\
|\cG|^2 \log |\cG|, & \text{if $|\cG|< p^{1/2}\log p$}.
\end{array}
\right. $$
\end{cor}

We define $N(\cF,\cG, \cH)$ to be the number of triples of solutions to $f_1(g_1-g_2)=f_2(h_1-h_2)$ where $f_i \in  \cF, g_i \in \cG, h_i \in \cH$ for $i=1,2$. Using Corollary \ref{cor:MultEnergy} we obtain the following result.
\begin{lemma} \label{lem:hectuples}
Let $\cF, \cG, \cH$ be multiplicative subgroups of $\F^*_p$ with cardinalities $F, G, H$ respectively with $G\ge H$. Additionally, let $M=\max(F,G)$. Then 
\begin{align*}
N(\cF,\cG, \cH) \ll  \frac{F^2}{M^{1/2}} \left\{
\begin{array}{ll}
G^2H^2p^{-1/2},& \text{if $ H \ge p^{1/2}\log p$},\\
G^2H^{3/2+o(1)}p^{-1/4}, & \text{if $G\ge p^{1/2}\log p>H$,}\\
(GH)^{3/2+o(1)},& \text{if $ G < p^{1/2}\log p$.}
\end{array}
\right.
\end{align*}
\end{lemma}
\begin{proof}
By multiplying both sides of $f_1(g_1-g_2)=f_2(h_1-h_2)$ by the inverses $f_2^{-1}$ and $h_2^{-1}$ and taking a factor of $g_2$ from the left hand side, and defining $S=\{fgh : f \in \cF, g \in \cG, h \in \cG\}$ we have
\begin{align*}
N(\cF,\cG, \cH)& = \frac{F^2GH}{|S|} \sum_{\lambda \in S} |\{\lambda(g-1)=h-1 : g \in \cG, h \in \cH\}|.
\end{align*}
By the Cauchy inequality,
\begin{align*}
N(\cF,\cG, \cH)^2 &\le \frac{F^4G^2H^2}{|S|} \left| \left\{ \frac{h_1-1}{g_1-1}=\frac{h_2-1}{g_2-1} : h_i \in \cH, g_i \in \cG, i=1,2\right\} \right| \\
&= \frac{F^4G^2H^2}{|S|} (\rE^\times (\cG-1)\rE^\times (\cH-1))^{1/2}.
\end{align*}
By Corollary \ref{cor:MultEnergy},
\begin{align*}
N(\cF,\cG, \cH)^2 \ll  \frac{F^4G^2H^2}{|S|} \left\{
\begin{array}{ll}
G^2H^2/p,& \text{if $ H \ge p^{1/2}\log p$},\\
G^2H^{1+o(1)}p^{-1/2}, & \text{if $G\ge p^{1/2}\log p>H$,}\\
(GH)^{1+o(1)},& \text{if $ G < p^{1/2}\log p$.}
\end{array}
\right.
\end{align*}
Since $|S|\ge M$ we complete our proof.
\end{proof}

Applying Lemma \ref{Bound Dx2} and Lemma \ref{lem:hectuples} in the proof of \cite[Theorem 1.4]{PetShp}, we obtain the following result on quadrilinear sums over subgroups.

\begin{lemma} \label{lem:Bound T4}
For any multiplicative subgroups $\cW, \cX, \cY, \cZ \subseteq \F_p^*$ of cardinalities $W, X, Y, Z$, respectively, with 
$W \ge X \ge Y \ge  Z$
and weights 
$\vartheta=(\vartheta_{w,x,y})$, $\rho= (\rho_{w,x,z})$,  $\sigma = (\sigma_{w,y,z})$ and $\tau=(\tau_{x,y,z})$
with 
$$
\max_{(w,x,y) \in \cW \times \cX \times \cY} |\vartheta_{w,x,y}| \le 1, \quad 
\max_{(w,x,y) \in \cW \times \cX \times \cZ} |\rho_{w,x,z}| \le 1, \quad
$$
$$
\max_{(w,x,y) \in \cW \times \cY \times \cZ} |\sigma_{w,y,z}| \le 1, \quad 
\max_{(w,x,y) \in \cX \times \cY \times \cZ} |\tau_{x,y,z}| \le 1,
$$
for the sums
\begin{align*}
T=\sum_{w \in\cW} \sum_{x \in \cX} \sum_{y \in \cY} \sum_{z\in \cZ} \vartheta_{w,x,y}\rho_{w,x,z} \sigma_{w,y,z} \tau_{x,y,z} \ep(a wxyz)
\end{align*} 
we have
\begin{align*}
|T|&\ll WXZY^{7/8}\\
& \qquad + \left\{
\begin{array}{ll}
W^{31/32}XYZp^{-1/32},& \text{if $ Y\ge p^{1/2}\log p$},\\
W^{31/32}XY^{15/16+o(1)}Z,& \text{if $ X \ge p^{1/2}\log p>Y$}, \\
W^{31/32}(XY)^{15/16+o(1)}Zp^{1/32},& \text{if $ W \ge p^{1/2}\log p>X$}, \\
W^{29/32+o(1)}(XY)^{15/16}Zp^{1/16},& \text{if $ W < p^{1/2}\log p$}.
\end{array}
\right.
\end{align*}
uniformly over $a\in \F_p^*$.
\end{lemma}
\begin{proof}
We see from \cite[p. 24]{PetShp} that
\begin{align*}
|T|^8 \ll (WXY)^6Z^7 \sum_{\mu \in \F^*_p} \sum_{\lambda \in \F_p} J(\mu) I(\lambda)\eta_\mu \ep(\lambda\mu) +(WXZ)^8Y^7,
\end{align*}
where $\eta_\mu$, $\mu \in \F^*_p$ is a complex number with $|\eta_\mu|=1$, $J(\mu)$ is the number of quadruples $(x_1,x_2,y_1,y_2) \in \cX^2\times \cY^2$ such that $(x_1-x_2)(y_1-y_2)=\mu \in \F^*_p$ and $I(\lambda)$ is the number of triples $(w_1,w_2,z)\in \cW^2 \times \cZ$ such that $z(w_1-w_2) =\lambda \in \F_p$.
We estimate $J(\mu)$ as in \cite[Equation 3.10]{PetShp} but using our bound from Lemma \ref{Bound Dx2} to obtain
\begin{align} \label{eq:J}
\sum_{\mu \in \F^*_p} J(\mu)^2 \ll \left\{
\begin{array}{ll}
X^4Y^4/p , & \text{if $ Y \ge p^{1/2}\log p$},\\
X^4 Y^{3+ o(1)}p^{-1/2} , & \text{if $ X \ge p^{1/2}\log p>Y$},\\
(XY)^{3+ o(1)}, & \text{if $X< p^{1/2}\log p$}. 
\end{array}
\right.
\end{align}
Now
\begin{align*}
&\sum_{\lambda \in \F_p}I(\lambda)^2 \\
&\quad =   |\{z_1(w_1-w_2)=z_2(w_3-w_4) : w_1,w_2 \in \cW, z_i \in \cZ, i=1,2,3,4 \}|\\
&\quad =N(Z,W,W).
\end{align*}
Therefore, by Lemma \ref{lem:hectuples},
\begin{align} \label{eq:I}
&\sum_{\lambda \in \F_p}I(\lambda)^2 \ll  \left\{
\begin{array}{ll}
Z^{2}W^{7/2}p^{-1/2},& \text{if $ W\ge p^{1/2}\log p$},\\
Z^{2}W^{5/2+o(1)},& \text{if $ W < p^{1/2}\log p$.}
\end{array}
\right.
\end{align}
Applying the classical bound on bilinear exponential sums from Lemma \ref{lem:bilin} together with \eqref{eq:J} and \eqref{eq:I}, we get
\begin{align*}
|T|^8 \ll &(WXZ)^8Y^7 \\
& + \left\{
\begin{array}{ll}
W^{31/4}X^8Y^8Z^8p^{-1/4},& \text{if $ Y\ge p^{1/2}\log p$},\\
W^{31/4}X^8Y^{15/2+o(1)}Z^{8},& \text{if $ X \ge p^{1/2}\log p>Y$}, \\
W^{31/4}X(YZ)^{15/2+o(1)}Z^8p^{1/4},& \text{if $ W \ge p^{1/2}\log p>X$}, \\
W^{29/4+o(1)}(XY)^{15/2}Z^8p^{1/2},& \text{if $ W < p^{1/2}\log p$}.
\end{array}
\right.
\end{align*}
Hence,
\begin{align*} \label{eq:T}
|T|\ll &WXZY^{7/8}\\
& + \left\{
\begin{array}{ll}
W^{31/32}XYZp^{-1/32},& \text{if $ Y\ge p^{1/2}\log p$},\\
W^{31/32}XY^{15/16+o(1)}Z,& \text{if $ X \ge p^{1/2}\log p>Y$}, \\
W^{31/32}(XY)^{15/16+o(1)}Zp^{1/32},& \text{if $ W \ge p^{1/2}\log p>X$}, \\
W^{29/32+o(1)}(XY)^{15/16}Zp^{1/16},& \text{if $ W < p^{1/2}\log p$}.
\end{array}
\right.
\end{align*}
This completes the proof.
\end{proof} 
We compare our bound for subgroups from Lemma \ref{lem:Bound T4} with that for arbitrary sets coming from \cite[Theorem~1.4]{PetShp}
\begin{equation} \label{eq:oldbound 4}
\begin{split}
&\left|\sum_{w \in\cW} \sum_{x \in \cX} \sum_{y \in \cY} \sum_{z\in \cZ} \vartheta_{w,x,y}\rho_{w,x,z} \sigma_{w,y,z} \tau_{x,y,z} \ep(a wxyz)\right| \\ 
& \qquad \qquad \qquad \qquad \ll p^{1/16}W^{15/16}(XY)^{61/64}Z^{31/32}.
\end{split}
\end{equation}
For example, if $W=X=Y=Z=p^{1/2+o(1)}$ then the bounds become $p^{125/64+o(1)}$ and $p^{63/32+o(1)}$ respectively.

\section{Proof of Theorem \ref{thm:Bound3}}
Let $\cG_\alpha, \cG_\beta, \cG_\gamma$ be the subgroups of $\F^*_p$ formed by the elements of orders $\alpha, \beta$ and $\gamma$ respectively.
Then,
\begin{align*}
S_\chi(\Psi) &= \frac{1}{\alpha \beta \gamma} \sum_{x\in \cG_\alpha} \sum_{y\in \cG_\beta}\sum_{z\in \cG_\gamma}\sum_{w \in \F^*_p} \chi(wxyz)\ep (\Psi(wxyz)) \\
&= \frac{1}{\alpha \beta \gamma} \sum_{x\in \cG_\alpha} \sum_{y\in \cG_\beta}\sum_{z\in \cG_\gamma}\sum_{w \in \F^*_p} \\ 
&\qquad \chi(wxyz)\ep (aw^ky^kz^k+bw^\ell x^\ell z^\ell +cw^mx^my^m+dw^nx^ny^nz^n) \\
&=\frac{1}{\alpha \beta \gamma} \sum_{x\in \cG_\alpha} \sum_{y\in \cG_\beta}\sum_{z\in \cG_\gamma}\sum_{w \in \F^*_p} \vartheta_{w,x,y} \rho_{w,x,z} \sigma_{w,y,z}\ep (dw^nx^ny^nz^n)
\end{align*}
where $ \vartheta_{w,x,y}=\chi(wxy)\ep(cw^mx^my^m)$, $\rho_{w,x,z} =\chi(z)\ep (bw^\ell x^\ell z^\ell)$ and  $\sigma_{w,y,z}=\ep (aw^ky^kz^k)$. Now the image $\cW = \{ w^n : w\in \F^*_p\}$ of non-zero $n$th powers contains $(p-1)/\delta$ elements, each appearing with multiplicity $\delta$. Similarly, we can see that the images $\cX=\{x^n : x\in \cG_\alpha \}, \cY=\{y^n : y\in \cG_\beta \}$ and $\cZ=\{z^n : z\in \cG_\gamma \}$ contain $f, g$ and $h$ elements with multiplicity $\gcd(\alpha,\delta),  \gcd(\beta,\delta)$ and $\gcd(\gamma,\delta)$ respectively. We apply Lemma \ref{lem:Bound T4}, recalling our assumption that $f\ge g$ and noticing $f\delta=\lcm (\alpha,\delta)<p-1$, hence $f\le p/\delta$, which gives us
\begin{align*}
S_\chi(\Psi) \ll &\frac{\delta \gcd(\alpha,\delta)\gcd(\beta,\delta)\gcd(\gamma,\delta)}{\alpha \beta \gamma} (p/\delta)fg^{7/8}h\\
&+\frac{\delta \gcd(\alpha,\delta)\gcd(\beta,\delta)\gcd(\gamma,\delta)}{\alpha \beta \gamma} \\
& \times \left\{
\begin{array}{ll}
(p/\delta)^{31/32}fghp^{-1/32},& \text{if $ g\ge p^{1/2}\log p$},\\
(p/\delta)^{31/32}fg^{15/16+o(1)}h,& \text{if $ f \ge p^{1/2}\log p>g$}, \\
(p/\delta)^{31/32}(fg)^{15/16+o(1)}hp^{1/32},& \text{if $ p/\delta \ge p^{1/2}\log p>f$}, \\
(p/\delta)^{29/32+o(1)}(fg)^{15/16}hp^{1/16},& \text{if $ p/\delta < p^{1/2}\log p$}.
\end{array}
\right. \\
=& pg^{-1/8} \\
&+\left\{
\begin{array}{ll}
p^{15/16}\delta^{1/32},& \text{if $ g\ge p^{1/2}\log p$},\\
p^{31/32}\delta^{1/32}g^{-1/16+o(1)},& \text{if $ f \ge p^{1/2}\log p>g$}, \\
p\delta^{1/32}(fg)^{-1/16+o(1)},& \text{if $p/\delta  \ge p^{1/2}\log p>f$}, \\
p^{31/32+o(1)}\delta^{3/32}(fg)^{-1/16},& \text{if $ p/\delta < p^{1/2}\log p$}.
\end{array}
\right.
\end{align*}
This concludes the proof.

\end{document}